\documentclass[12pt,twoside, draft]{article}
\usepackage{ifthen}
\usepackage{amsmath}
\usepackage{amsfonts}
\usepackage{latexsym}
\usepackage{amsthm}
\usepackage{amssymb}
\usepackage{color}

\newtheorem{Theorem}{\quad Theorem}[section]
\newtheorem{question}{\quad Question}[section]
\newtheorem{remark}{\quad Remark}[section]

\newtheorem{Lemma}[Theorem]{\quad Lemma}

\newtheorem{example}[Theorem]{\quad Example}

\date{}

\begin{document}


\centerline{}

\centerline {\Large{\bf Entire Dirichlet series with monotonous}}
\centerline {\Large{\bf coefficients and logarithmic h-measure}}

\centerline{}

\centerline{\bf {S. I. Panchuk}}

\centerline{}

\centerline{Department of Mechanics and Mathematics,}

\centerline{Ivan Franko National University of L'viv, Ukraine}

\centerline{s.panchuk@lnu.edu.ua}

\centerline{}

\centerline{\bf {T. M. Salo}}

\centerline{}

\centerline{Institute of Applied Mathematics and Fundamental Sciences,}

\centerline{National University "Lvivs'ka Politekhnika", Lviv, Ukraine}

\centerline{tetyan.salo@gmail.com}

\centerline{}

\centerline{\bf {O. B. Skaskiv}}

\centerline{}

\centerline{Department of Mechanics and Mathematics,}

\centerline{Ivan Franko National University of L'viv, Ukraine}

\centerline{olskask@gmail.com}

\centerline{}

\begin{abstract}
Let $F$ be an entire function represented by absolutely convergent
for all $z\in\mathbb{C}$ Dirichlet series of the form $ F(z) =
\sum\nolimits_{n=0}^{+\infty} a_{n}e^{z\lambda_{n}},$\  where a
sequence $(\lambda_n)$ such that $\lambda_n\in\mathbb{R}\ \
(n\geq0)$, $\lambda_n\not=\lambda_k$ for any $n\not=k$ and $(\forall
n\geq 0):\ 0\leq\lambda_n<\beta:=\sup\{\lambda_j:\ j\geq0\}\leq
+\infty.$ {Let $h$  be non-decrease positive continuous function on
$[0,+\infty)$ and $\Phi$ increase positive continuous on
$[0,+\infty)$ function.}

 In this paper we {find} the condition {on} $(\mu_n)$ and $(\lambda_n)$ {such that}
the relation $F(x+iy)=(1+o(1))a_{\nu(x, F)}e^{(x+iy)\lambda_{\nu(x, F)}}
$ holds as
 $x\to +\infty$\ outside some set $E$ of finite logarithmic  $h$-measure uniformly in $y\in\mathbb{R}$.
\end{abstract}



\section{Introduction}
 Let
$\mathcal{L}$ be the class of positive continuous increasing
functions on $[0;+\infty)$  and $\mathcal{L}_{+}$  the subclass of
functions $\Phi\in\mathcal{L}$ such that $\Phi(t)\to +\infty$\
$(t\to +\infty)$. By $\varphi$ we denote inverse function to
$\Phi\in \mathcal{L}$.

Let $\mathcal D$ be a class of entire (absolutely convergent
in the whole complex plane $\mathbb{C}$) Dirichlet series of the
form
\begin{equation}\label{e1}
F(z) = \sum\limits_{n=0}^{+\infty} a_{n}e^{z\lambda_{n}},\quad
z\in\mathbb{C},
\end{equation}
where a sequence $(\lambda_n)$ such that $\lambda_n\in\mathbb{R}\
\ (n\geq0)$, $\lambda_n\not=\lambda_k$ for any $n\not=k$ and
\begin{equation}\label{obm}
    (\forall n\geq 0): \quad 0\leq\lambda_n<\beta:=\sup\{\lambda_j:\ j\geq0\}\leq +\infty.
\end{equation}
 For $F\in \mathcal D$ and $x\in\mathbb{R}$ we denote

\smallskip\centerline{$ M(x,F)=\sup\{|F(x+iy)|\colon
y\in\mathbb{R}\},\quad m(x,F)=\inf\{|F(x+iy)|\colon
y\in\mathbb{R}\},$}

\smallskip\noindent and by

\smallskip\centerline{$\mu(x,F)=\max\{|a_n|e^{x\lambda_n}\colon n\ge 0\},\
\nu(x,F)=\max\{n\colon|a_n|e^{x\lambda_n}=\mu(x,F)\}$}

\smallskip\noindent  the maximal term and central index of
series~\eqref{e1}, respectively.

By $\mathcal{D}_{a}$ we denote a subclass of Dirichlet series
$F\in\mathcal D$ with a fixed sequence $a=(|a_n|)$, $|a_n|\searrow
0$\ $(n_0\leq n\to +\infty)$, and for $\Phi\in \mathcal{L}$ by
$\mathcal{D}_{a}(\Phi)$  subclass of functions $F\in\mathcal{D}_{a}$
such that  $\ln\mu(x,F)\geq x\Phi(x)$\ $(x\geq x_0)$.

Let \[\mu_n:=-\ln|a_n|\hskip30pt (n\geq0).\]

In paper \cite{Ska94} one can find such theorem:

\medskip\noindent{\bf Theorem A (O.B. Skaskiv, 1994 \cite{Ska94}).} {\sl For every entire function $F\in \mathcal{D}_a$    relation
\begin{equation}\label{e3}
F(x+iy)=(1+o(1))a_{\nu(x, F)}e^{(x+iy)\lambda_{\nu(x, F)}}
\end{equation}
holds as $x\to +\infty$\ outside some set $E$ of finite logarithmic
measure, i.e. $\text{\rm log-meas}(E):=\int_{E}d\ln x<+\infty$,\
uniformly in $y\in\mathbb{R}$, if and only if
\begin{equation}\label{e4}
\sum\limits_{n=n_0}^{+\infty}\frac1{\mu_{n+1}-\mu_{n}}<+\infty.
\end{equation}
}

It is easy to see that relation \eqref{e3} holds for $x\to +\infty$\
$(x\notin E)$\  uniformly  in $y\in\mathbb{R}$, if and only if
\begin{equation}\label{osn1}
M(x,F)\sim \mu(x,F)\quad (x\to +\infty,\ x\notin E),
\end{equation}
hence it follows
\begin{equation}\label{osn2}
M(x,F)\sim m(x,F)\quad (x\to +\infty,\ x\notin E).
\end{equation}

The  finiteness of logarithmic measure of an exceptional set $E$ in
Theorem A is the sharp estimate. {It} follows from the {following}
theorem.

\medskip\noindent{\bf Theorem B (Ya.Z. Stasyuk, 2008 \cite{Sta}).} {\sl For every increasing sequence  $(\mu_n)$, such that condition
\eqref{e4} satisfies and for any function $h\in \mathcal{L}_+$ there
exist an entire Dirichlet series $F\in\mathcal{D}_a$ with
$|a_n|=\exp\{-\mu_n\}$,
 a set $E$ and a constant $d>0$ such that
 \begin{equation}\label{ner}
(\forall x\in E)\colon\quad   M(x,F)\geq(1+d) \mu(x, F),\quad
M(x,F)\geq (1+d) m(x,F)
 \end{equation}
and
$$
\text{\rm h-log-meas}(E):=\int_{E\cap[1, +\infty)}h(x) d\ln
x=+\infty.
$$}

Due to Theorem B the natural {\bf question} arises: {what conditions
must satisfy the entire Dirihlet series $F\in\mathcal{D}_a$ {such}
that relation \eqref{e3} holds for $x\to +\infty$\ outside some set
$E$ of finite  logarithmic h-measure,  i.e. $\text{\rm
h-log-meas}(E)<+\infty$?} In this article we give the answer to this
question. Our main result is the following.

\begin{Theorem}\label{the1} {\sl Let  $(\mu_n)$ be a sequence such that  condition \eqref{e4}
holds, $h\in\mathcal{L}_{+}, \Phi\in \mathcal{L}$ and
$F\in\mathcal{D}_{a}(\Phi)$. If
\begin{gather}\label{3mm}
(\forall\
b>0)\colon\quad\sum\limits_{n=n_0}^{+\infty}h\left(\varphi(\lambda_{n})\cdot\Big(1+\frac{b}{\mu_{n+1}-\mu_{n}}\Big)\right)\dfrac{1}{\mu_{n+1}-\mu_{n}}<+\infty,
\end{gather}
then the relation \eqref{e3} holds as
 $x\to +\infty$\ outside some set $E$ of finite logarithmic  $h$-measure uniformly in $y\in\mathbb{R}$.}
\end{Theorem}
The method of proof of Theorem \ref{the1} differs from the method of
proofs corresponding statements in \cite{Ska94,Ska} and is close to
the methods of proofs from papers \cite{SkaSal,LutsSka,SalSka15}.


\section{Proof of Theorem \ref{the1}}
 We denote $\Delta_0=0$ and for $n\geq 1
$
\begin{equation}\label{delta}
\Delta_n :=\delta\cdot\sum_{j=0}^{n-1}\left(\mu_{j+1}-\mu_j\right)
\sum_{m=j+1}^{+\infty}\left(\frac{1}{\mu_{m}-\mu_{m-1}}+\frac{1}{\mu_{m+1}-\mu_m}\right),\quad\delta
>0.
\end{equation}
It easy to see that
\begin{equation}\label{delta2}
\Delta_n\geq n\delta\quad (n\geq 0),\qquad \Delta_n=o(\mu_n)\quad
(n\to +\infty).
\end{equation}
We put $b_n=e^{\lambda_n}$  and consider the Dirichlet series
\[f(s)=\sum_{n=0}^{+\infty}b_n\ e^{s\mu_n},\quad\mu_n=-\ln|a_n|.\]
The condition
$\sum\limits_{n=0}^{+\infty}1/{(\mu_{n+1}-\mu_{n})}<+\infty$ implies
$n^2=o(\mu_n)$\ $(n\to +\infty)$  (see \cite {Ska94,Ska}), thus $\ln
n =o(\mu_n)$ $(n\to+\infty)$. $F\in\mathcal{D}_a$ so
 $\lim\limits_{n\to+\infty}\dfrac{-\ln
|a_n|}{\lambda_n}=+\infty.$ Therefore by Valiron's formulae for the
abscissa of absolute convergence (\cite[p.115]{Leont}) we get
\begin{gather*}\sigma_{\text{abs}}(f)=\varliminf\limits_{n\to+\infty}\frac{-\ln
b_n}{\mu_n}=\lim\limits_{n\to+\infty}\Big(\frac{\ln
|a_n|}{\lambda_n}\Big)^{-1}=0.
\end{gather*}
Now we consider Dirichlet series
\[{f}_q(s)=\sum_{n=0}^{+\infty}\frac{b_n}{\alpha^q_n}\ e^{s\mu_n},\quad \alpha_n=e^{\Delta_n},\ q\in\mathbb{R}.  \]
From the second relation in \eqref{delta2} and the condition
$\sigma_{\text{abs}}(f)=0$ it follows that
$$
\sigma_{\text{abs}}(f_q)=\varliminf\limits_{n\to+\infty}\frac{-\ln
b_n+q\Delta_n}{\mu_n}=\sigma_{\text{abs}}(f)+q\cdot\lim\limits_{n\to+\infty}\frac{\Delta_n}{\mu_n}=0
$$
for any $q\in\mathbb{R}$. Therefore the Dirichlet series of the form
\begin{gather*}
{f}^*(s)=\sum_{n=0}^{+\infty}{b_n} e^{s\mu_n^*},\quad
\mu_n^*=\mu_n+\Delta_n,
\end{gather*}
is absolutely convergent in the whole half-plane
$\Pi_0:=\{s\colon\text{\rm Re }s<0\}$ and $\sigma_{\text{abs}}(f^*)=
0.$ Indeed, for every fixed $s\in\Pi_0$ we have as $ q=-\text{\rm Re
}s$
\begin{gather*}
\Big|{b_n} e^{s\mu_n^*}\Big|= \Big|\frac{b_n}{\alpha_n^q}
e^{s\mu_n}\Big|\quad (\forall\ n\geq 0).
\end{gather*}
But $\sigma_{\text{abs}}(f_q)=0,$ thus $\sigma_{\text{abs}}(f^*)\geq
0.$ In the other hand ${b_n}
e^{s\mu_n^*}\Big|_{s=0}=e^{\lambda_n}\not\to 0$\ $(n\to +\infty)$,
hence $\sigma_{\text{abs}}(f^*)= 0.$

From condition \eqref{obm} (see proof of Lemma 2 in
\cite[p.121--122]{Ska94}) we get
$$
\nu(x, f^*)\to +\infty\quad (x\to-0).
$$

We need 
the following lemma (compare, for example,
\cite{SkaSal,LutsSka,SalSka15}).
\begin{Lemma}\label{lealpha} {\sl For all $n\geq0$ and $k\geq1$   inequality
\begin{equation}\label{5}
   \frac{\alpha_n}{\alpha_k}\ e^{\tau_k(\mu_n-\mu_k)}\leq e^{-\delta |n-k|}
\end{equation}
holds, where
\[\tau_k:=t_k+\dfrac{\delta}{\mu_k-\mu_{k-1}},
\hskip30pt t_k:=\dfrac{\Delta_{k-1}-\Delta_{k}}{\mu_k-\mu_{k-1}}.\]
}
\end{Lemma}

\begin{proof}[Proof of Lemma \ref{lealpha}]
We remark that
\begin{gather}
t_k=-\delta\cdot\sum\limits_{m=k}^{+\infty}\left(\dfrac{1}{\mu_m-\mu_{m-1}}+\dfrac{1}{\mu_{m+1}-\mu_m}\right),
\label{tau2}\\
\tau_k=-2\delta\cdot\sum\limits_{m=k+1}^{+\infty}\dfrac{1}{\mu_m-\mu_{m-1}},
\label{tau}\\
\tau_{k+1}-\tau_{k}=\dfrac{2\delta}{\mu_{k+1}-\mu_{k}}\quad (k\geq1).\label{tau3}
\end{gather}

Since \[\ln\alpha_n-\ln\alpha_{n-1}=\Delta_n-\Delta_{n-1}=-t_n(\mu_n-\mu_{n-1}),\]for $n\geq k+1$  we have \[\ln\dfrac{\alpha_n}{\alpha_k}+\tau_k(\mu_n-\mu_k)=
-\sum\limits_{j=k+1}^{n}t_j(\mu_j-\mu_{j-1})+\tau_k\sum\limits_{j=k+1}^{n}(\mu_j-\mu_{j-1})=\]
\[=-\sum\limits_{j=k+1}^{n}\left(t_j-\tau_k\right)\left(\mu_j-\mu_{j-1}\right)\leq
-\sum\limits_{j=k+1}^{n}\left(t_j-\tau_{j-1}\right)\left(\mu_j-\mu_{j-1}\right)=\]
\[=-\sum\limits_{j=k+1}^{n}\delta=-(n-k)\cdot\delta.\]
Similarly, for $n\leq k-1$ we obtain
\[
\ln\dfrac{\alpha_n}{\alpha_k}+\tau_k(\mu_n-\mu_k)=-\ln\dfrac{\alpha_k}{\alpha_n}-\tau_k(\mu_k-\mu_n)=
\]
\[
=\sum\limits_{j=n+1}^{k}t_j(\mu_j-\mu_{j-1})-\tau_k\sum\limits_{j=n+1}^{k}(\mu_j-\mu_{j-1})=
-\sum\limits_{j=n+1}^{k}\left(\tau_k-t_j\right)\left(\mu_j-\mu_{j-1}\right)\leq
\]
\[
\leq-\sum\limits_{j=n+1}^{k}\left(\tau_{j}-t_j\right)\left(\mu_j-\mu_{j-1}\right)=-\sum\limits_{j=n+1}^{k}\delta=-(k-n)\cdot\delta
\]
and Lemma 1 is proved.
\end{proof}

We remark that from definitions of $\tau_{k}$ and $t_{k}$ (see
\eqref{tau}, \eqref{tau2}) and the condition
$\sum\limits_{k=0}^{+\infty}1/(\mu_{k+1}-\mu_{k})<+\infty$ it follows that
there exists $k_0(\delta)$ such that
\[\tau_{k}\geq -1\quad (k\geq k_0(\delta)),\qquad \tau_{k} <0\quad
(k\geq 1).\]

Let $J$ be a set of the values of the central index $\nu(\sigma,
f^*)$, i.e.
$$
J=\big\{k\in\mathbb{N}\colon (\exists\ \sigma <0)[\nu(\sigma,
f^*)=k]\big\}.
$$
{Denote by} $(R_{k})$  a sequence of the points of the springs of
$\nu(\sigma, f^*)$, enumerate such that $\nu(\sigma, f^*)=k$ for
$\sigma\in[R_{k},R_{k+1})$ in the case $R_{k}<R_{k+1}$. Then for
$\sigma\in[R_{k},R_{k+1})$, $k\in J$ from Lemma \ref{lealpha} we
have
 \[{b_n}e^{\sigma\mu_n^*}\leq{b_{k}}e^{\sigma\mu_{k}^*}\Longleftrightarrow\
 \frac{b_ne^{\sigma\mu_n}}{b_ke^{\sigma\mu_{k}}}\leq\dfrac{\alpha_n^{|\sigma|}}{\alpha_{k}^{|\sigma|}}\leq e^{-|\sigma|\tau_{k}(\mu_n-\mu_{k})} e^{-|\sigma||n-k|\cdot\delta}\]
for all $n\geq0$. Hence,
\[ \frac{b_ne^{(\sigma +|\sigma|\tau_{k})\mu_n}}{b_ke^{(\sigma +|\sigma|\tau_{k})\mu_{k}}}\leq e^{-|\sigma||n-k|\cdot\delta}\]
i.e., for all $k\in J$ and for every
$\sigma^*\in\big[R_{k}(1+|\tau_{k}|),R_{k+1}(1+|\tau_{k}|)\big)$
\[\dfrac{b_ne^{\sigma^*\mu_n}}{b_ke^{\sigma^*\mu_{k}}}\leq \exp\Big\{-\frac{|\sigma^*||n-k|\cdot\delta}{1+|\tau_{k}|}\Big\}.\]
Thus, as
$x=\dfrac{1}{|\sigma^*|}\in\left[-\dfrac{1}{R_{k}(1+|\tau_{k}|)},-\dfrac{1}{R_{k+1}(1+|\tau_{k}|)}\right)$
we get
\[\dfrac{|a_n|e^{x\lambda_n}}{|a_{k}|e^{x\lambda_{k}}}\leq \exp\Big\{-\frac{|n-k|\cdot\delta}{1+|\tau_{k}|}\Big\}\]
for all $k\in J$, $n\geq0$. Therefore,
\begin{equation}\label{nu}
\nu(x,F)=k,\ \mu(x, F)=|a_{k}|e^{x\lambda_{k}},\
x\in\left[-\frac{1}{R_{k}(1+|\tau_{k}|)},-\frac{1}{R_{k+1}(1+|\tau_{k}|)}\right).
\end{equation}

Denote
\[
E^*(\delta):=\bigcup\limits_{k=k_0(\delta)}^{+\infty}\left[-\dfrac{1}{R_{k}(1+|\tau_{k}|)},-\dfrac{1}{R_{k+1}(1+|\tau_{k}|)}\right),\\ \
E(\delta):=[0, +\infty)\setminus E^*(\delta).\]
Then for every $x>0,\ x\notin E(\delta)$
\begin{gather}
    |F(x+iy)-a_{\nu(x,F)}e^{(x+iy)\lambda_{\nu(x,F)}}|\leq\nonumber\\
    \leq \mu(x,F)\cdot\sum_{n\neq\nu(x,F)}\exp\Big\{-\frac{\delta\cdot |n-\nu(x,F)|}{1+|\tau_{\nu(x,F)}|}\Big\}\leq
    \frac{2e^{-\delta/2}}{1-e^{-\delta/2}}\cdot\mu(x,F) \label{osnovna}
\end{gather}
because $1+|\tau_{\nu(x,F)}|<2$\  ($x\in E^*(\delta)$). It remains to
prove that the logarithmic $h-$measure of a set $E(\delta)$ is
finite. Using
\begin{gather}E(\delta)\subset[0, x_0)\bigcup\Big(
\bigcup\limits_{k=k_0(\delta)+1}^{+\infty}\Big[-\frac{1}{R_{k}(1+|\tau_{k-1}|)},-\frac{1}{R_{k}(1+|\tau_{k}|)}\Big)\Big), \nonumber\\
x_0=-\frac{1}{R_{k_0(\delta)}(1+|\tau_{k_0(\delta)-1}|)},\nonumber
\end{gather}
we obtain
\begin{gather*}
\text{\rm h-log-meas}(E\cap
[x_0,+\infty))=\sum\limits_{k=k_0(\delta)+1}^{+\infty}\int\limits_{\big[-\frac{1}{R_{k}(1+|\tau_{k-1}|)},-\frac{1}{R_{k}(1+|\tau_{k}|)}\big)}h(x)d\ln
x\leq\nonumber\\
\leq\sum\limits_{k=k_0(\delta)+1}^{+\infty}h\left(-\frac{1}{R_{k}(1+|\tau_{k}|)}\right)\ln\dfrac{1+|\tau_{k-1}|}{1+|\tau_{k}|}=\nonumber\\
=\sum\limits_{k=k_0(\delta)+1}^{+\infty}h\left(-\frac{1}{R_{k}(1+|\tau_{k}|)}\right)\ln\Big(1+\dfrac{|\tau_{k-1}|-|\tau_{k}|}{1+|\tau_{k}|}\Big)\leq
\nonumber\\
\leq
\cdot\sum\limits_{k=k_0(\delta)}^{+\infty}h\left(-\frac{1}{R_{k+1}(1+|\tau_{k+1}|)}\right)(|\tau_{k}|-|\tau_{k+1}|),
\end{gather*}

Hence, using equality \eqref{tau3} we have
\begin{gather}\label{meas}
\text{\rm h-log-meas}(E\cap [x_0,+\infty))\leq
2\delta\cdot\sum\limits_{k=k_0(\delta)}^{+\infty}h\left(-\frac{1}{R_{k+1}(1+|\tau_{k+1}|)}\right)\dfrac{1}{\mu_{k+1}-\mu_{k}}.
\end{gather}
The condition $F\in\mathcal{D}_a(\Phi)$ implies
$$
x\Phi(x)\leq{\ln\mu(x,F)}=-\mu_{\nu(x-0,F)}+x\lambda_{\nu(x-0,F)}\leq x\lambda_{\nu(x-0,F)}\quad
(x\geq x_1>0),
$$
therefore
\begin{equation}\label{last}
x\leq\varphi(\lambda_{\nu(x-0,F)})\quad (x\geq x_1>0).
\end{equation}
Denote $\theta_{k}:=-\dfrac{1}{R_{k+1}(1+|\tau_{k}|)}$. By \eqref{nu}
we have $\nu(\theta_{k}-0)=k$, thus from
\eqref{tau3} and \eqref{last} it follows
\begin{gather}-\frac{1}{R_{k+1}(1+|\tau_{k+1}|)}=\theta_{k}\cdot\frac{1+|\tau_{k}|}{1+|\tau_{k+1}|}=
\theta_{k}\cdot\Big(1+\frac{|\tau_{k}|-|\tau_{k+1}|}{1+|\tau_{k+1}|}\Big)\leq\nonumber\\
\leq\theta_{k}\cdot\Big(1+\frac{2\delta}{\mu_{k+1}-\mu_{k}}\Big)\leq\varphi(\lambda_{k})\cdot\Big(1+\frac{2\delta}{\mu_{k+1}-\mu_{k}}\Big)\label{teta}
\end{gather}
for all $k\geq k_1(\delta)$. Using inequality \eqref{teta} to
 inequality \eqref{meas}, we get
\begin{gather}
\text{\rm h-log-meas}(E(\delta)\cap [x_0,+\infty))\leq\nonumber \\
\leq 2\delta\cdot\sum\limits_{k=k_0(\delta)}^{k_2(\delta)-1}h\left(-\frac{1}{R_{k+1}(1+|\tau_{k+1}|)}\right)\dfrac{1}{\mu_{k+1}-\mu_{k}}+\nonumber \\
+2\delta\cdot\sum\limits_{k=k_2(\delta)}^{+\infty}h\left(\varphi(\lambda_{k})\cdot\Big(1+\frac{2\delta}{\mu_{k+1}-\mu_{k}}\Big)\right)\dfrac{1}{\mu_{k+1}-\mu_{k}}:=K(\delta)<+\infty,
\label{meas2}
\end{gather}
where $k_2(\delta)=\max\{k_0(\delta),k_1(\delta)\}$.  Relation
\eqref{meas2} implies that
\begin{gather*}
(\forall\ \delta>0)\colon\quad\text{\rm h-log-meas}(E(\delta)\cap
[x,+\infty))=o(1)\quad (x\to +\infty).
\end{gather*}
We put now $\delta_n=n, \varepsilon_n=2^{-n}$\ $(n\geq 1)$. Then for
every $n\geq 1$ there exists $x_n\geq x_0$ such that
\begin{gather*}
\text{\rm h-log-meas}(E(\delta_n)\cap [x_n,+\infty))\leq
\varepsilon_n.
\end{gather*}
Without loss of generality we may assume that $x_n<x_{n+1}$\ $(n\geq
1)$. Denote $
E=\bigcup\limits_{n=1}^{+\infty}\big(E(\delta_n)\cap[x_n;x_{n+1})\big)
$.  Define a function $\gamma\colon [x_1,+\infty)\to [0,+\infty)$ by
equality $\gamma(x)=\dfrac{2e^{-\delta_n/2}}{1-e^{-\delta_n/2}}$ for
$x\in[x_n, x_{n+1})$. Then from inequality \eqref{osnovna} it
follows
\begin{gather}\label{end}
|F(x+iy)-a_{\nu(x,F)}e^{(x+iy)\lambda_{\nu(x,F)}}|\leq
\gamma(x)\cdot\mu(x,F)
\end{gather}
for all $x\in [x_1, +\infty)\setminus E$ uniformly in
$y\in\mathbb{R}$. But $\gamma(x)\to 0$\ $(x\to +\infty)$ and
$$
\text{\rm h-log-meas}(E\cap
[x_1,+\infty))\leq\sum_{n=1}^{+\infty}\text{\rm
h-log-meas}(E(\delta_n)\cap [x_n,x_{n+1}))\leq
\sum_{n=1}^{+\infty}\varepsilon_n=1.
$$
Thus, $ \text{\rm h-log-meas}(E)<+\infty.$

\section{Consequences and concluding remarks}

\begin{remark}\rm
Since $\lambda_n/\mu_n\to 0 $ $(n\to +\infty)$, we have $\lambda_n
<\mu_n$ for all $n$ large enough. So $\lambda_{n}$ {one can replace}
with $\mu_{n}$ in \eqref{3mm}.
\end{remark}

In the case $\beta=\sup\{\lambda_j:\ j\geq0\}= +\infty$ condition
\eqref{3mm} of Theorem \ref{the1} can be written in a simpler form.

Let $\Phi\in\mathcal{L}_+$ and
$\mathcal{D}^{*}_{a}(\Phi):=\bigcup\limits_{\rho>0}\mathcal{D}_{a}(\Phi_{\rho})$,
$\Phi_{\rho}(x):=\rho\cdot\Phi(x\rho)$.

\begin{Theorem}\label{the2} {\sl Let  $(\mu_n)$ be a sequence such that condition \eqref{e4}
holds, $h\in\mathcal{L}_{+}, \Phi\in \mathcal{L}_+$ and
$F\in\mathcal{D}^{*}_{a}(\Phi)$. If
\begin{gather}\label{3mm'}
(\forall\
b>0)\colon\qquad\sum\limits_{n=n_0}^{+\infty}\dfrac{h\left(b\varphi(b\lambda_{n})\right)}{\mu_{n+1}-\mu_{n}}<+\infty,
\end{gather}
then relation \eqref{e3} holds as
 $x\to +\infty$\ outside some set $E$ of finite logarithmic  $h$-measure uniformly in $y\in\mathbb{R}$.}
\end{Theorem}

\begin{remark}\rm In the case $\Phi(x)=e^x/x$ we obtain that
$\mathcal{D}^{*}_{a}(\Phi)$ is the class Dirichlet series of nonzero
lower R-order
$$
\varliminf_{x\to +\infty}\frac{\ln\ln M(x,F)}{x}:=\rho_R[F]\in
(0,+\infty]
$$
and  condition \eqref{3mm'} from condition
\begin{gather}\label{3mm''}
(\forall\
b>0)\colon\qquad\sum\limits_{n=n_0}^{+\infty}\dfrac{h\left(b\ln\lambda_{n}\right)}{\mu_{n+1}-\mu_{n}}<+\infty
\end{gather}
follows.
\end{remark}

\begin{example}\rm Well known, if $\ln n=o(\lambda_n\ln\lambda_n)$\ and $\ln\lambda_{n+1}\sim\ln\lambda_n$ $(n\to +\infty)$ then for the function $F\in D$  with coefficients
$$
|a_n|=\exp\Bigl\{-\frac 1{\rho}\lambda_n\ln\lambda_n\Bigl\}
$$
we have $\rho_R[F]=\rho$. Thus condition \eqref{3mm''} follows from
the condition
\begin{gather*}
(\forall\
b>0)\colon\qquad\sum\limits_{n=n_0}^{+\infty}\dfrac{h\left(b\ln\lambda_{n}\right)}{\lambda_{n+1}\ln\lambda_{n+1}-\lambda_{n}\ln\lambda_{n}}<+\infty.
\end{gather*}
\end{example}

\begin{question}\rm  {Is the description of exceptional sets in our Theorems
\ref{the1} and \ref{the2}  the best possible?}
\end{question}
\begin{question}\rm {Are conditions \eqref{3mm} and \eqref{3mm'}  in our Theorems
\ref{the1} and \ref{the2} necessary?}
\end{question}

\end{document}